\documentclass[a4paper, notitlepage, oneside, reqno]{amsart}
\usepackage{amsmath,amsfonts,amssymb}
\usepackage{amsthm} 
\usepackage{graphicx} 
\usepackage{subfig}
\usepackage{listings}
\usepackage{geometry}
\usepackage{hyperref}
\usepackage{wrapfig}
\usepackage{url}
\usepackage{tikz}
\usetikzlibrary{shapes, positioning, arrows}
\usepackage{eso-pic}
\usepackage{mathrsfs}
\usepackage{xcolor}
\usepackage{mathtools}

\DeclarePairedDelimiter\floor{\lfloor}{\rfloor}

\newcommand{\bigO}{\mathcal O}
\newcommand{\diff}{{\mathrm D}}

\newcommand{\dx}{\, \mathrm{d}x}
\newcommand{\dy}{\, \mathrm{d}y}

\newcommand{\R}{\mathbb{R}}
\newcommand{\Z}{\mathbb{Z}}
\newcommand{\N}{\mathbb{N}}

\newcommand{\Si}{\mathbb{S}}

\DeclareMathOperator{\F}{{\mathscr F}}

\DeclareMathOperator{\id}{id}

\newtheorem{theorem}{Theorem}[section]
\newtheorem{lemma}[theorem]{Lemma}

\newtheorem{proposition}[theorem]{Proposition}

\numberwithin{equation}{section}
\theoremstyle{remark}
\newtheorem{remark}[theorem]{Remark}

\title[Waves of maximal height for the Degasperis-Procesi equation]{A non-local approach to waves of maximal height for the Degasperis-Procesi equation}
\author{Mathias Nikolai Arnesen}
\date{\today}

\begin{document}

\begin{abstract}
We consider the non-local formulation of the Degasperis-Procesi equation $u_t+uu_x+L(\frac{3}{2}u^2)_x=0$, where $L$ is the non-local Fourier multiplier operator with symbol $m(\xi)=(1+\xi^2)^{-1}$. We show that all $L^\infty$, pointwise travelling-wave solutions are bounded above by the wave-speed and that if the maximal height is achieved they are peaked at those points, otherwise they are smooth. For sufficiently small periods we find the highest, peaked, travelling-wave solution as the limiting case at the end of the main bifurcation curve of $P$-periodic solutions. The results imply that the Degasperis-Procesi equation does not admit cuspon solutions.
\end{abstract}

\maketitle

\section{Introduction}
We consider the equation
\begin{equation}
\label{eq: DP non-local}
u_t+uu_x+(L(\frac{3}{2}u^2))_x =0, \quad x\in \R, \,\, t\in \R,
\end{equation}
where $u$ is a scalar function and $L$ is the nonlocal operator $L=(1-\partial_x^2)^{-1}$. That is,
\begin{equation*}
Lf=K\ast f, \quad K=\F^{-1}m,
\end{equation*}
where $m(\xi)=(1+\xi^2)^{-1}$ and $\F$ denotes the Fourier transform. Equation \eqref{eq: DP non-local} is the nonlocal formulation of the Degasperis--Procesi equation \cite{Degasperis1999ai}
\begin{equation}
\label{eq: DP local}
u_t-u_{xxt}+4u u_x-3u_x u_{xx}-u u_{xxx}=0,
\end{equation}
as can easily be seen by applying the inverse operator of $L$, $1-\partial_x^2$, to \eqref{eq: DP non-local}. The Degasperis--Procesi equation was discovered as one of three equations within a certain class of third order PDEs satisfying an asymptotic integrability condition up to third order, the other two being the KdV and the Camassa--Holm equations \cite{Degasperis1999ai}. Like these two equations, the Degasperis--Procesi equation has a lax pair, a bi-Hamiltonian structure, and an infinite number of conservation laws \cite{Degasperis2002ani}. While it was discovered purely for its mathematical properties, it has later been rigorously derived as a model for the propagation of shallow water waves, having the same asymptotic accuracy as the Camassa--Holm equation \cite{Constantin2009thr}. The Degasperis--Procesi and Camassa--Holm equations feature stronger nonlinear effects than the KdV equation (or rather, the dispersion is much weaker), making them better suited to modelling nonlinear phenomena like wave breaking and solutions with singularities, while maintaining the rich mathematical structure mentioned above that other weakly-dispersive models like the Whitham equation \cite{Whitham1967vma} lack. 

Shortly after its discovery, the well-posedness of \eqref{eq: DP non-local} was extensively studied, establishing that it is locally well-posed in $H^s$ both on $\R$ and $\Si$ for $s>3/2$, and admitting both global classical and weak solutions and classical solutions that blow up in finite time \cite{Yin2003otc}, \cite{Yin2003gef}, \cite{Yin2004gws}. Moreover, the blow-up only occurs as wave-breaking. That is, the solution remains bounded, but it's slope goes to $- \infty$; for a detailed study of the blow-up for \eqref{eq: DP non-local}, see \cite{Escher2006gws} and references therein. 

The weak dispersion allows not only for wave-breaking, but also for waves with singularities in the form of sharp crests at the wave-peaks. Indeed, explicit peaked soliton solutions, as well as multipeakon solutions which are not travelling waves, to \eqref{eq: DP local} are known \cite{Degasperis2002ani}. These are of the same form as the ones for Camassa-Holm equation \cite{Camassa1993ais}, and indeed every equation in the so-called 'b-family' of equations that the Degasperis-Procesi and Camassa-Holm equations belong to has such solutions \cite{Degasperis2002ani}.

In this paper we will focus on travelling-wave solutions to \eqref{eq: DP non-local}. Assuming $u(x,t)=\varphi(x-\mu t)$ is a travelling wave, where $\mu\in \R$ is the wave-speed, \eqref{eq: DP non-local} takes the form
\begin{equation}
\label{eq: DP}
-\mu \varphi+\frac{1}{2}\varphi^2+\frac{3}{2}L(\varphi^2)=a,
\end{equation}
where $a\in \R$ is a constant of integration. By a Galilean change of variables this is equivalent to $-\mu\varphi +\frac{1}{2}\varphi^2+\frac{3}{2}L(\varphi^2+k\varphi)=0$, where $k$ depends on $\mu$ and $a$; in particular, $k\neq 0$ for $a\neq 0$. Hence there is no Galilean change of variables that removes $a$ while preserving the form of the equation. We will work with the equation in the form \eqref{eq: DP}.

From the structure of the equation it is readily deducible that all non-constant solutions to \eqref{eq: DP} are smooth except potentially at points where the wave-height equals the wave-speed (cf. Theorem \ref{thm: regularity I} or \cite{Lenells2005tws}) and singularities can only occur in the form of sharp crests with height equal to the wave-speed. We therefore call such solutions for waves of \emph{maximal height}. In this paper we will study the regularity and existence of travelling waves of maximal height to \eqref{eq: DP} from a nonlocal perspective.

The motivation of this paper is two-fold: to provide novel information about waves of maximal height for the DP equation specifically and to better understand the formation of highest waves and their singularities for nonlinear dispersive equations more generally. We therefore consider the non-local formulation and follow the general framework of \cite{Ehrnstrom2016owc} and \cite{Ehrnstrom2016eoa}. We show firstly that any even, non-constant $L^\infty$ solution of \eqref{eq: DP} is peaked wherever the maximal height is achieved. That is, it is Lipschitz continuous at the crest(s), but not $C^1$. In particular this means that there are no cuspon solutions of \eqref{eq: DP} in $L^\infty$. The restriction to bounded to solutions is quite natural as while equation \eqref{eq: DP} makes sense for any $\varphi\in H^{-2}(\R)$, if we exclude purely distributional solutions, any function solving \eqref{eq: DP} a.e. clearly belongs to $L^\infty$. Secondly, for sufficiently small periods peaked solutions of \eqref{eq: DP} are found as the limiting case at the end of the main bifurcation curve of $C_{\text{even}}^\alpha(\Si_P)$ solutions for $\alpha\in (1,2)$. While it has been established that there are peaked periodic travelling-wave solutions to \eqref{eq: DP local} for all non-zero wave speeds in \cite{Lenells2005tws}, the approach of that paper works only for the local formulation and cannot be extended to a genuinely non-local equation. Moreover the methods \cite{Lenells2005tws} and that used here are entirely different and give different insight and information.

As $L^\infty$ cuspon solutions to \eqref{eq: DP local} have been claimed by several authors, our claim that they do not exists requires some comment. The cuspons are invariably found studying the local equation, as they cannot appear in the non-local formulation as we show in this paper, and they are strong solutions in all points except the cusps. The exclusion of the cusps is crucial, however. Consider for instance the stationary cusped soliton $u(x)=\sqrt{1-\mathrm{e}^{-2|x|}}$ discovered in \cite{Zhang2007cas}, which is a pointwise solution to \eqref{eq: DP local} at all points except $0$, where the function has a cusp. For any test function $\varphi\in C_0^\infty(\R)$, treating the left-hand side of \eqref{eq: DP local} as a distribution (note that $u$ is independent of time), one can with basic calculus show that
\begin{equation*}
\langle 4u u_x-3u_x u_{xx}-u u_{xxx}, \varphi\rangle =\langle u^2, \frac{1}{2}\varphi_{xxx}-2\varphi_x\rangle=\int_\R u^2\left(\frac{1}{2}\varphi_{xxx}-2\varphi_x\right)\dx=2\varphi_x(0)
\end{equation*}
and hence it is not a weak solution to \eqref{eq: DP local}, but rather to
\begin{equation*}
u_t-u_{xxt}+4u u_x-3u_x u_{xx}-u u_{xxx}=-2\delta',
\end{equation*}
where $\delta$ is the usual delta-distribution. This is the case with all cuspons of the DP equation - there are point mass distributions at the cusps. To accept any function that solves the equation pointwise at all but a countable number of points as a solution is equivalent to claiming that the sawtooth function $u(x)=x-\text{floor}(x)$, or indeed any piece-wise linear function, is a solution to the equation
\begin{equation*}
u''(x)=0, \quad x\in \R,
\end{equation*}
which is clearly absurd. Hence we think it more correct to call the cuspons solutions not of \eqref{eq: DP local} with $0$ right-hand side, but with some point mass distributions.


The paper is structured as follows: first some essential properties of the operator $L$ and its kernel $K$ are recounted in Section \ref{sec: L and K}. In Section \ref{sec: D-P} we establish some general results about solutions to \eqref{eq: DP} and, in particular, using the properties of $K$, study the behaviour around points of critical height and prove Theorem \ref{thm: regularity II}, stating that any even, nonconstant solution is peaked at points where $\varphi=\mu$. Lastly, in Section \ref{sec: GB} we use the bifurcation Theory of \cite{Buffoni2003ato} to construct a global bifurcation curve of even, periodic solutions in $C^\alpha$ for $\alpha\in (1,2)$. Using the properties of solutions established in Section \ref{sec: D-P}, we show that for sufficiently small periods the solutions along the curve converge to an even, non-constant solution that achieves the maximal height and must therefore be a peakon.

\section{The operator \(L\) and its kernel}
\label{sec: L and K}
As $\widehat{L f}(\xi)=(1+\xi^2)^{-1}\widehat{f}(\xi)$, $Lf$ can formally be expressed as a convolution
\begin{equation*}
Lf(x)=K\ast f (x)=\int_\R K(x-y)f(y)\dy,
\end{equation*}
where $K(x)$ is the inverse Fourier transform of $m(\xi)$. In this case, an explicit expression is well known from virtually any textbook on Fourier analysis:
\begin{equation}
\label{eq: K}
K(x)=\F^{-1}((1+\xi^2)^{-1})=\frac{1}{2}\mathrm{e}^{-|x|}. 
\end{equation}
In particular, we note that $K$ is completely monotone on $(0,\infty)$; it is positive, strictly decreasing and strictly convex for $x>0$.

The periodic kernel is
\begin{equation*}
K_P(x)=\sum_{n\in \Z}K(x+nP),
\end{equation*}
for $P\in (0,\infty)$. For $x\in (-P/2, P/2)$, $K(x+nP)=\frac{1}{2}\mathrm{e}^{-|x+nP|}=\frac{1}{2}\mathrm{e}^{-x}\mathrm{e}^{-nP}$ for $n\geq 1$, and $K(x+nP)=\frac{1}{2}\mathrm{e}^{x}\mathrm{e}^{nP}$ for $n\leq -1$. Thus
\begin{align}
K_P(x) & =\sum_{n\in \Z}K(x+nP) \nonumber \\
&=\frac{1}{2}\mathrm{e}^{-|x|}+\frac{1}{2}(\mathrm{e}^x+\mathrm{e}^{-x})\sum_{n=1}^\infty \mathrm{e}^{-nP} \nonumber \\
& =\frac{1}{2}\mathrm{e}^{-|x|}+\cosh(x)\frac{1}{\mathrm{e}^{P}-1}. \label{eq: K_P}
\end{align}
For periodic functions, the operator $L$ is given by $Lf(x)=\int_{-P/2}^{P/2}K_P(x-y)f(y)\dy$.

We conclude this section with a rather obvious, but crucial lemma:
\begin{lemma}
\label{lem: monotone}
$L$ is strictly monotone: $Lf>Lg$ if $f$ and $g$ are bounded and continuous functions with $f\gneq g$.
\end{lemma}
\begin{proof}
Let $f$ and $g$ be as in the statement of the lemma. As $K$ is strictly positive, we get that for all $x\in \R$, $K(x-\cdot)(f-g)\gneq 0$ and by continuity strictly positive on a set of non-zero measure. Hence
\begin{equation*}
Lf(x)-Lg(x)= \int_\R K(x-y)(f(y)-g(y))\dy>0.
\end{equation*}
Clearly, the same argument holds for $K_p$.
\end{proof}

\section{Periodic travelling waves}
\label{sec: D-P}

Note that if $\varphi(x)$ is a travelling wave solution to \eqref{eq: DP non-local} with wave-speed $\mu$, then $-\varphi(-x)$ is also a travelling solution to \eqref{eq: DP non-local} with wave-speed $-\mu$. We will therefore only consider $\mu>0$.

First we investigate how the parameter $a$ in \eqref{eq: DP} influences the behaviour/existence of solutions.
\begin{theorem}
\label{thm: solutions to DP}
Fix $\mu>0$ and $P<\infty$. For all values of $a\in \R$, non-constant $P$-periodic solutions to \eqref{eq: DP} (if they exist) satisfy
\begin{equation*}
\min \varphi  <\frac{\mu+\sqrt{\mu^2+8a}}{4} <\max\varphi .
\end{equation*}
Moreover,
\begin{itemize}
\item[(i)] For $a\leq 0$, all solutions are non-negative. When $a<-\frac{\mu^2}{8}$ there are no real solutions and for $a=-\frac{\mu^2}{8}$ there is only the constant solution $\varphi=\frac{\mu}{4}$,
\item[(ii)] there are only constant solutions when $a\geq \mu^2$.
\end{itemize}
\end{theorem}

\begin{proof}
At any point $x$ where $\varphi(x)^2=L(\varphi^2)(x)=:R$, \eqref{eq: DP} reduces to
\begin{equation*}
R(2R-\mu)= a, 
\end{equation*}
which has the positive solution $R=\frac{\mu+\sqrt{\mu^2+8a}}{4}$. As $L(c)=c$ for constants and $L$ is strictly monotone (Lemma \ref{lem: monotone}), there has to exist points where $\varphi^2<L(\varphi^2)$ and points where $\varphi^2>L(\varphi^2)$ for non-constant $P$-periodic solutions $\varphi$. Thus the first inequality has to hold if $\max\varphi >|\min \varphi  |$. 

Consider first the case $a\leq 0$. Then $\varphi$ cannot be negative in any point as then the left-hand side of \eqref{eq: DP} would be strictly positive in that point ($Lf$ is non-negative if $f$ is non-negative). Let $m=\min \varphi\geq 0$. Then $L(\varphi^2)\geq m^2$ with equality if and only if $\varphi\equiv m$. Hence, if $\varphi$ is a solution to \eqref{eq: DP}, we get
\begin{equation*}
m(2m-\mu)\leq a.
\end{equation*}
For $a<-\frac{\mu^2}{8}$ this has no real solutions, and for $a=-\frac{\mu^2}{8}$ this has only the constant solution $\varphi=\frac{\mu}{4}$. This proves (i).

Now let $a>0$. Assume that $\varphi<0$ on some intervals. By Theorem \ref{thm: regularity I}, $\varphi$ is smooth on these intervals. Clearly, $\varphi$ is bounded below, so there is a point $x_0$ such that $\varphi(x_0)=\min  \varphi $. Then $L(\varphi \varphi')(x_0)=0$ and $L(\varphi^2)$ attains it minimum at $x_0$. This implies that $\varphi$ also has to be positive at some point, and $M:=\max \varphi> |\min \varphi  |$. Thus the first inequality holds and $M>\frac{\mu+\sqrt{\mu^2+8a}}{4}$. In particular, this means that $\max  \varphi \geq \frac{\mu}{2}$ for all $a\geq 0$ and $M>\sqrt{a}$ if $a<\mu^2$. We have that
\begin{equation}
\label{eq: rewritten DP}
(\varphi-\mu)^2=\mu^2+2a-3L(\varphi^2).
\end{equation}
Assume $a\geq \mu^2$. Note that if $\varphi=\mu$ at any point, then $3L(\varphi^2)=\mu^2+2a\geq 3\mu^2$ at those points. If $a=\mu^2$, then the constant solution $\varphi\equiv \mu$ is a valid solution, otherwise Lemma \ref{lem: monotone} implies that $\varphi$ must also take values above $\mu$. Assume $\varphi \gneq \mu$ is a non-constant solution. Then the left-hand side of \eqref{eq: rewritten DP} attains its minimum where $\varphi$ is attains its minimum, while the right-hand side attains its minimum where $L(\varphi^2)$ attains its maximum. This is a contradiction. As both $K$ and $K_P$ are even and completely monotone on $(0,\infty)$ and $(0,P/2)$, respectively, $L(\varphi^2)$ cannot be maximal where $\varphi^2$ is minimal.

Assume now that $\varphi$ takes values both above and below $\mu$. Then $L(\varphi^2)(x)$ is maximal whenever $\varphi(x)=\mu$ and $3L(\varphi^2)(x)=\mu^2+2a$ these points. Moreover, $3L(\varphi^2)<\mu^2+2a$ when $\varphi>\mu$. This implies that there are infinitely many disjoint intervals, each of finite length, where $\varphi>\mu$, and that $L(\varphi^2)$ has its minimum on each interval at the points where $\varphi$ is maximal. This is again not possible.
\end{proof}

Henceforth we will assume that $a$ is such that non-constant solutions exists, i.e. that $-\mu^2/8<a<\mu^2$.
\begin{theorem}
\label{thm: strictly increasing}
Let $P(0,\infty]$. Any $P$-periodic, non-constant and even solution $\varphi\in BUC^1(\R)$ that is non-decreasing on $(-P/2,0)$ satisfies
\begin{equation*}
\varphi'>0, \,\, \varphi<\mu \,\, \text{on} \,\, (-P/2,0).
\end{equation*}
If $\varphi\in BUC^2(\R)$, then
\begin{equation*}
\varphi''(0)<0, \quad \text{and} \quad \varphi(0)<\mu,
\end{equation*}
and if $P<\infty$, then
\begin{equation*}
\varphi''(\pm P/2)>0.
\end{equation*}
\end{theorem}
\begin{proof}
Let $\varphi$ be a non-constant and even solution that is non-decreasing on $(-P/2,0)$. We can rewrite \eqref{eq: DP} as $(\mu-\varphi)^2=\mu^2+2a-3L(\varphi^2)$, and if $\varphi\in BC^1(\R)$ we can differentiate on each side to get
\begin{equation}
\label{eq: first der}
(\mu-\varphi(x))\varphi'(x)=\frac{3}{2}L(\varphi^2)'(x).
\end{equation}
As $\varphi$ is even, $\varphi'$ will be odd and using the evenness of $K_P$ we get
\begin{equation}
\label{eq: L of varphi'varphi}
L(\varphi^2)'(x)=2\int_{-P/2}^0 (K_P(x-y)-K_P(x+y))\varphi'(y)\varphi(y)\dy.
\end{equation}
We claim that $K_P(x-y)>K_P(x+y)$ for any $x,y\in (-P/2,0)$. Fix $x\in (-P/2,0)$. As $K_P$ strictly decreases from the origin in $(-P/2,P/2)$ and is $P$-periodic, the claim follows if, for all $y\in (-P/2,0)$,
\begin{equation*}
|x-y|<\min \lbrace |x+y|, x+y+P\rbrace.
\end{equation*}
As $x$ and $y$ are same-signed, we have
\begin{equation*}
|x+y|=|x|+|y|>|x-y|.
\end{equation*}
Moreover, we have $-x<P+x$ and $-y<P+y$, so $|x-y|=\max\lbrace x-y, y-x\rbrace<P+x+y$. This proves the claim.

Now we claim that $\varphi'\varphi\gneq 0$ on $(-P/2,0)$. By assumption $\varphi'\gneq 0$ on this interval. If $a\leq 0$, it is plain to see that $\varphi>0$ as the right hand side of \eqref{eq: DP} is strictly positive whenever $\varphi\leq 0$. For a $BUC^1(\R)$ solution the same is true when $0<a<\mu^2$ too; this follows from equation (4.14) in \cite{Lenells2005tws}. Hence the integrand in \eqref{eq: L of varphi'varphi} is non-negative and strictly positive on a set of positive measure in $(-P/2,0)$, and it follows that the right hand side of \eqref{eq: first der} is strictly positive for $x\in (-P/2,0)$. This implies the first part of the statement.

Assume now that $\varphi\in BUC^2(\R)$. Then we can differentiate each side of \eqref{eq: first der} to get
\begin{equation}
\label{eq: second der}
(\mu-\varphi(x))\varphi''(x)-(\varphi'(x))^2=\frac{3}{2}L(\varphi^2)''(x)=3L(\varphi''\varphi+(\varphi')^2)(x).
\end{equation}
Evaluating this at $x=0$ and using the evenness of $\varphi''$, $\varphi$, $(\varphi')^2$ and $K_P$, we get
\begin{align*}
(\mu-\varphi(0))\varphi''(0)= & 6\int_{-P/2}^0 K_P(y)\left(\varphi''(y)\varphi(y)+\varphi'(y)^2\right)\dy \\
= & 6\left[K_P(y)\varphi'(y)\varphi(y)\right]_{y=-P/2}^{y=0} -6\int_{-P/2}^0 K_P'(y)\varphi'(y)\varphi(y)\dy.
\end{align*}
The first term on the second line vanishes as $\varphi'(0)=\varphi'(-P/2)=0$. If $P=\infty$, then $\lim_{y\rightarrow -\infty}K(y)=0$ and we get the same conclusion. As $K_P$ is strictly increasing on $(-P/2,0)$, we get that the final integral is strictly positive. That is,
\begin{equation*}
(\mu-\varphi(0))\varphi''(0)=-6\int_{-P/2}^0 K_P'(y)\varphi'(y)\varphi(y)\dy<0
\end{equation*}
As we already proved that $\varphi<\mu$ on $(-P/2,0)$, it is not possible that $\varphi(0)>\mu$, and thus we conclude that $\varphi(0)<\mu$ and $\varphi''(0)<0$.

Now we assume that $P<\infty$. Note that $K_P(-P/2-y)=K_P(P/2-y)=K_P(-P/2+y)=K_P(P/2+y)$. Evaluating \eqref{eq: second der} at $x=-P/2$, we get
\begin{align*}
(\mu-\varphi(-P/2))\varphi''(-P/2)= & 6\int_{-P/2}^0 K_P(P/2+y)\left(\varphi''(y)\varphi(y)+\varphi'(y)^2\right)\dy \\
= & \left[K_P(P/2+y)\varphi'(y)\varphi(y)\right]_{y=-P/2}^{y=0}-6\int_{-P/2}^0 K_P'(P/2+y)\varphi'(y)\varphi(y)\dy.
\end{align*}
As above, the first term in the second line vanishes. As $K_P$ is strictly decreasing on $(0,P/2)$, we get that $K_P'(P/2+y)<0$ for $y\in (-P/2,0)$, and it follows that the last integral is negative. That is,
\begin{equation*}
(\mu-\varphi(-P/2))\varphi''(-P/2)=-6\int_{-P/2}^0 K_P'(P/2+y)\varphi'(y)\varphi(y)\dy>0,
\end{equation*}
and we conclude that $\varphi''(-P/2)>0$.
\end{proof}

\subsection{Singularity at \(\varphi=\mu\)}
Now we investigate what happens as a solution approaches $\mu$ from below. First we show that a solution is smooth below $\mu$:
\begin{theorem}
\label{thm: regularity I}
Let $\varphi\leq \mu$ be a solution of \eqref{eq: DP}. Then:
\begin{itemize}
\item[(i)] If $\varphi<\mu$ uniformly on $\R$, then $\varphi\in C^{\infty}(\R)$ and all of its derivatives are uniformly bounded on $\R$.
\item[(ii)] If $\varphi<\mu$ uniformly on $\R$ and $\varphi\in L^2(\R)$, then $\varphi\in H^\infty(\R)$.
\item[(iii)] $\varphi$ is smooth on any open set where $\varphi<\mu$.
\end{itemize}
\end{theorem}

\begin{proof}
Assume first that $\varphi<\mu$ uniformly on $\R$. Note that as $\varphi\rightarrow -\infty$, the left-hand side of \eqref{eq: DP} goes to $\infty$, hence $\varphi$ must be bounded below as well. Clearly, $|m^{(n)}(\xi)|\lesssim (1+|\xi|)^{-2-n}$ (that is, $m$ is a $S^{-2}$-multiplier) and $L$ is therefore continuous from the Besov space $B_{p,q}^s(\R)$ to $B_{p,q}^{s+2}(\R)$ for all $s\in\R$ and $1\leq p,q\leq \infty$. Denoting by $\mathcal{C}^s(\R)$, $s\in \R$ the Zygmund space $B_{\infty,\infty}^s(\R)$, we have in particular that $L$ maps $L^\infty(\R)\subset B_{\infty,\infty}^0(\R)$ into $\mathcal{C}^2(\R)$, and therefore $\varphi\mapsto L(\varphi^2)$ maps $L^{\infty}(\R)$ into $\mathcal{C}^2(\R)$. Recall that if $s\in \R_{+}\setminus \N$, then $\mathcal{C}^s(\R)=C^s(\R)$, the ordinary H\"older space, and if $s\in \N$ then $W^{s,\infty}(\R)\subsetneq \mathcal{C}^s(\R)$. 

As $\varphi$ solves \eqref{eq: DP} we have
\begin{equation*}
(\varphi-\mu)^2=\mu^2+2a-3L(\varphi^2).
\end{equation*}
The assumption $\varphi<\mu$ therefore implies that $3L(\varphi^2)<\mu^2+2a$, and the operator $L(\varphi^2)\mapsto \mu -\sqrt{\mu^2+2a-3L(\varphi^2)}$ therefore maps $B_{p,q}^s(\R)\cap L^\infty(\R)$ into itself for $s>0$. Since $\varphi<\mu$, we also get that $\mu -\sqrt{\mu^2+2a-3L(\varphi^2)}=\varphi$. Combining this map with $\varphi\mapsto L(\varphi^2)$ and iterating, we get (i). When $p=q=2$, $B_{p,q}^s(\R)$ can be identified with $H^s(\R)$. Assume now that $\varphi\in L^2(\R)$ in addition. As $\varphi$ is also bounded, we get that $\varphi^2\in L^2(\R)\cap L^\infty(\R)$, and in general $\varphi^2\in H^s(\R)\cap L^\infty(\R)$ if $\varphi\in H^s(\R)\cap L^\infty$, and thus $\varphi\mapsto L(\varphi^2)$ maps $H^s(\R)\cap L^{\infty}(\R)$ to $H^{s+2}(\R)\cap L^{\infty}(\R)$, and we can apply the above iteration argument again. This proves (ii).

Lastly, to prove (iii), we note that if $\varphi\in L^\infty(\R)$ and $\mathcal{C}_{loc}^s$ on an open set $U$ in the sense that $\psi\varphi\in \mathcal{C}^s(\R)$ for any $\psi\in C_0^\infty(U)$, we still get that $L(\varphi)$ is $\mathcal{C}_{loc}^{s+2}$ in $U$ (the proof of this is the same as in Theorem 5.1 \cite{Ehrnstrom2016owc}). Thus we can apply the same iteration argument as above again.
\end{proof}

The next lemma will be essential for showing that the global bifurcation curves do not converge to a trivial case.

\begin{lemma}
\label{lem: global bound}
Let $P<\infty$, and let $\varphi$ be an even, non-constant solution of \eqref{eq: DP} that is non-decreasing on $(-P/2,0)$ with $\varphi \leq \mu$. Then there exists a universal constant $C_{K,P,\mu}>0$, depending only on the kernel $K$ and the period $P$ and $\mu>0$, such that
\begin{equation*}
\mu-\varphi(\frac{P}{2})\geq C_{K,P,\mu}.
\end{equation*}

\end{lemma}
\begin{proof}
If $\varphi(-P/2)=\varphi(P/2)<0$, the statement is true with $C_{K,P,\mu}=\mu$. Assume therefore that $\varphi$ is non-negative. From the evenness and periodicity of $K_P$ and $\varphi$, we get the formula
\begin{align}
L(\varphi^2)(x+h)&-L(\varphi^2)(x-h) \nonumber \\
= & \int_{-P/2}^0 (K_P(x-y)-K_P(x+y))(\varphi(y+h)^2-\varphi(y-h)^2)\dy. \label{eq: double symmetrisation}
\end{align}
As $\varphi\geq 0$ is non-decreasing, both factors in the integrand are non-negative for $x\in (-P/2,0)$ and $h\in (0,P/2)$. We also have the equality
\begin{equation}
\label{eq: difference of solution}
(2\mu-\varphi(x)-\varphi(y))(\varphi(x)-\varphi(y))=3\left(L(\varphi^2)(x)-L(\varphi^2)(y)\right),
\end{equation}
which shows that $L(\varphi^2)(x)=L(\varphi^2)(y)$ whenever $\varphi(x)=\varphi(y)$. As $\varphi$ is assumed to be non-constant and non-negative, this identity together with \eqref{eq: double symmetrisation} implies that $\varphi$ is strictly increasing on $(-P/2,0)$, and it therefore follows from Theorem \ref{thm: regularity I} that $\varphi$ is smooth away from $x=kP$, $k\in \Z$.
Let $x\in \left[-\frac{3P}{8},-\frac{P}{8}\right]$. Then for a solution $\varphi$ as in the assumptions,
\begin{equation*}
(\mu-\varphi(\frac{P}{2}))\varphi'(x)\geq (\mu-\varphi(x))\varphi'(x)=\frac{3}{2}\lim_{h\rightarrow 0}\frac{L(\varphi^2)(x+h)-L(\varphi^2)(x-h)}{4h}.
\end{equation*}
 As the integrand in \eqref{eq: double symmetrisation} is non-negative for $h\in (0,P/2)$ and non-positive for $h\in (-P/2,0)$, we can apply Fatou's lemma to the limit above and we get
\begin{align*}
(\mu-\varphi(\frac{P}{2}))\varphi'(x) &\geq 3\int_{-P/2}^{P/2} K_P(x-y)\varphi(y)\varphi'(y)\dy \\
&=3\int_{-P/2}^0 (K_P(x-y)-K_P(x+y))\varphi(y)\varphi'(y)\dy.
\end{align*}
Assume for a contradiction that the statement is not true. Then for all $k<\mu$ there must exist a solution $\varphi$ satisfying the assumptions and such that $k\leq \varphi \leq \mu$. Then $\mu-\varphi(P/2)<\mu-k$. On the other hand, as $K_P(x-y)>K_P(x+y)$ for $x,y\in (-P/2,0)$, we get that
\begin{align*}
(\mu-\varphi(\frac{P}{2}))\varphi'(x) &\geq 3\int_{-P/2}^0 (K_P(x-y)-K_P(x+y))\varphi(y)\varphi'(y)\dy \\
& \geq 3k\int_{-3P/8}^{-P/8}(K_P(x-y)-K_P(x+y))\varphi'(y)\dy.
\end{align*}
There is a universal constant $\tilde{\lambda}_{K,P}>0$ depending only on $K_P$ and $P<\infty$ such that
\begin{equation*}
\min \lbrace K_P(x-y)-K_P(x+y): x,y\in \left[-\frac{3P}{8},-\frac{P}{8}\right]\rbrace\geq \tilde{\lambda}_{K,P}.
\end{equation*}
Integrating both sides above over $x\in \left(-\frac{3P}{8},-\frac{P}{8}\right)$, we get that
\begin{equation*}
(\mu-\varphi(\frac{P}{2}))(\varphi(-P/8)-\varphi(-3P/8))\geq 3k \frac{P}{8}\tilde{\lambda}_{K,P}(\varphi(-P/8)-\varphi(-3P/8)).
\end{equation*}
As shown above $\varphi$ is strictly increasing on $(-P/2,0)$, so $\varphi(-P/8)>\varphi(-3P/8)$ and we may divide out $(\varphi(-P/8)-\varphi(-3P/8))$ on both sides to get
\begin{equation*}
(\mu-\varphi(\frac{P}{2}))\geq 3k \frac{P}{8}\tilde{\lambda}_{K,P}.
\end{equation*}
This implies that $\mu-k\geq 3k \frac{P}{8}\tilde{\lambda}_{K,P}$ for all $k<\mu$. Taking the limit $k\nearrow \mu$, we get a contradiction.
\end{proof}

Now we come to the main result of this section, concerning the regularity at the point where $\varphi=\mu$.
\begin{theorem}
\label{thm: regularity II}
Let $\varphi\leq \mu$ be a solution of \eqref{eq: DP} which is even, non-constant, and non-decreasing on $(-P/2,0)$ with $\varphi(0)=\mu$. Then:
\begin{itemize}
\item[(i)] $\varphi$ is smooth on $(-P,0)$.
\item[(ii)] $\varphi\in C^{0,1}(\R)$, i.e. $\varphi$ is Lipschitz.
\item[(iii)] $\varphi$ is exactly Lipschitz at $x=0$; that is, there exists constants $0<c_1<c_2$ such that
\begin{equation*}
c_1|x|\leq |\mu-\varphi(x)| \leq c_2|x|
\end{equation*}
for $|x|\ll 1$.
\end{itemize}
\end{theorem}
\begin{proof}
Part (i) will follow directly from Theorem \ref{thm: regularity I} if we can show that $\varphi<\mu$ on $(-P/2,0)$. Assume that $x_0\in (-P/2,0]$ is the smallest number such that $\varphi(x_0)=\mu$; as $\varphi$ is assumed to be non-constant, it must be the case that $x_0>-P/2$. Then $\varphi(x)=\mu$ and $L(\varphi^2)'(x)=0$ for $x\in [x_0,0]$. That is,
\begin{equation*}
\int_{-P/2}^0\left(K_P'(x-y)+K_P'(x+y)\right)(\varphi(y))^2\dy=0, \quad x\in [x_0,0].
\end{equation*}
Clearly, $\int_{-P/2}^0\left(K_P'(x-y)+K_P'(x+y)\right)\dy=0$, and as
\begin{align*}
K_P'(x-y)+K_P'(x+y) & < 0, \quad -P/2<y<x<0, \\
K_P'(x-y)+K_P'(x+y) & > 0, \quad -P/2<x<y<0,
\end{align*}
we get that $\int_{-P/2}^x\left(K_P'(x-y)+K_P'(x+y)\right)\dy=-\int_{x}^0\left(K_P'(x-y)+K_P'(x+y)\right)\dy$. Hence, by the mean value theorem for integrals,
\begin{align*}
L(\varphi^2)'(x_0)= &\int_{-P/2}^0\left(K_P'(x_0-y)+K_P'(x_0+y)\right)(\varphi(y))^2\dy \\
= & \varphi(c)^2\int_{-P/2}^{x_0}\left(K_P'(x_0-y)+K_P'(x_0+y)\right)\dy \\
& +\mu^2\int_{x_0}^0\left(K_P'(x_0-y)+K_P'(x_0+y)\right)\dy \\
= &\int_{x_0}^0\left(K_P'(x_0-y)+K_P'(x_0+y)\right)\dy(\mu^2-\varphi(c)^2),
\end{align*}
for some $c\in (-P/2,x_0)$. As $-\mu<\varphi<\mu$ on $(-P/2,0)$, we get $(\mu^2-\varphi(c)^2)>0$, which is contradiction unless $\int_{x_0}^0\left(K_P'(x_0-y)+K_P'(x_0+y)\right)\dy=0$. That can only happen if $x_0=0$. This proves part (i).

At any point $x_0$ where $\varphi(x_0)=\mu$, \eqref{eq: difference of solution} reduces to
\begin{equation}
\label{eq: equality at x_0}
(\varphi(x_0)-\varphi(x))^2=3\left((L(\varphi^2)(x_0)-L(\varphi^2)(x)\right).
\end{equation}
From \eqref{eq: equality at x_0}, we get in the real line case that
\begin{align}
(\varphi(0)-\varphi(x))^2 & =\frac{3}{2}\int_\R (2K(y)-K(x+y)-K(x-y))(\varphi(y))^2\dy \nonumber \\
& \leq \frac{3}{2}\int_{|y|<|x|} (2K(y)-K(x+y)-K(x-y))(\varphi(y))^2 \dy \nonumber \\
& \leq \frac{3}{2}\|\varphi\|_{L^\infty(\R)}^2 \int_{|y|<|x|} |2K(y)-K(x+y)-K(x-y)|\dy, \label{eq: lipschitz proof}
\end{align}
where we used that the first integral on the right-hand side is clearly non-negative, while $2K(y)-K(x+y)-K(x-y)<0$ when $|y|\geq |x|$. Indeed, for $|y|>|x|$ we can expand $K(y+x)$ and $K(y-x)$ around $y$ and use the Lagrange remainder to get
\begin{equation*}
2K(y)-K(x+y)-K(x-y)=-\frac{x^2}{2}(K''(\xi_1)+K''(\xi_2))<0,
\end{equation*}
where $\xi_1\in(y,y+x)$, $\xi_2\in(y-x,y)$ and the last inequality follows from the strict convexity of $K$.

Similarly, expanding to one less order, we get
\begin{equation*}
2K(y)-K(x+y)-K(x-y)=x(K'(\xi_1)-K'(\xi_2)).
\end{equation*}
As $K'$ is uniformly bounded, there is a constant $C$ that can be chosen independently of $x$ such that
\begin{equation}
\label{eq: bound on K}
|2K(y)-K(x+y)-K(x-y)|\leq C|x|,
\end{equation}
for all $y\in \R$. Taking the square root on each side of \eqref{eq: lipschitz proof} we then get that
\begin{equation*}
|\varphi(0)-\varphi(x)|\leq C'\|\varphi\|_{L^\infty(\R)}|x|=C'\mu |x|.
\end{equation*}
This proves that $\varphi$ is Lipschitz at $0$. For the periodic kernel, we have that $2K_P(y)-K_P(x+y)-K_P(x-y)<0$ when $|x|\leq |y|\leq P/2-|x|$ (we are only interested in $x$ close to $0$, so we can assume $|x|<P/2-|x|$). In the intervals  $|y|<|x|$ and $P/2-|x|<|y|\leq P/2$, \eqref{eq: bound on K} holds for $K_P$ and we therefore get the same result.

It remains to show the opposite inequality, i.e. that $|\mu-\varphi(x)|\gtrsim |x|$ near $x=0$; in particular this implies that $\varphi\not\in C^1$. As $\varphi$ is smooth on $(-P/2,0)$ and (at least) Lipschitz in $0$, we can use integration by parts for $x\in(-P/2,0)$ to get
\begin{align*}
(\mu-\varphi(x))\varphi'(x)= & \frac{3}{2}L(\varphi^2)'(x) \\
= &\frac{3}{2}\int_{-P/2}^0\left(K_P'(x-y)+K_P'(x+y)\right)(\varphi(y))^2\dy\\
= & 3\int_{-P/2}^0 \left(K_P(x-y)-K_P(x+y)\right)\varphi(y)\varphi'(y)\dy.
\end{align*}
As $\mu-\varphi(x)\leq C'\mu |x|$ for $x\in (-P/2,0)$ as shown above, we divide out $\mu-\varphi(x)$:
\begin{equation*}
\varphi'(x) \geq C\int_{-P/2}^0 \frac{K_P(x-y)-K_P(x+y)}{|x|}\varphi'(y)\varphi(y)\dy,
\end{equation*}
for some constant $C>0$ independent of $x$. Let $x\in (-P/2,0)$. By the mean value theorem,
\begin{equation}
\label{eq: mean value thm}
\frac{|\mu-\varphi(x)|}{|x|}=\varphi'(\xi)\geq  C\int_{-P/2}^0 \frac{K_P(\xi-y)-K_P(\xi+y)}{|\xi|}\varphi'(y)\varphi(y)\dy
\end{equation}
for some $\xi\in(x,0)$. It suffices to show that this is bounded below by a positive constant as $x\nearrow 0$, but while $\varphi'$ is defined for all $x\in(-P/2,0)$, the limit may not exist. We therefore consider the limit infimum. On the other hand, the limit of the integral on the right hand side exists. Indeed, we have that
\begin{equation*}
\lim_{\xi\nearrow 0} \frac{K_P(\xi-y)-K_P(\xi+y)}{|\xi|}=2K_P'(y)
\end{equation*}
This function is non-negative and strictly monotonically increasing on $(-P/2,0)$, and as $\varphi$ is non-decreasing on this interval, we get by Lebesgue's dominated convergence theorem that for any sequence $\lbrace \xi_n\rbrace_n\subset (-P/2,0)$ such that $\xi_n\rightarrow 0$,
\begin{align*}
\lim_{n\rightarrow \infty} & C\int_{-P/2}^0 \frac{K_P(\xi_n-y)-K_P(\xi_n+y)}{|\xi_n|}\varphi'(y)\varphi(y)\dy \\
= &C\int_{-P/2}^0 \lim_{n\rightarrow \infty}\frac{K_P(\xi_n-y)-K_P(\xi_n+y)}{|\xi_n|}\varphi'(y)\varphi(y)\dy \\
\geq & C''\int_{-P/2}^0 \varphi'(y)\varphi(y)\dy \\
= &\frac{C''}{2}(\mu^2-(\varphi(-P/2))^2)>0.
\end{align*}
In particular the limit exists and therefore equals the limit infimum and from \eqref{eq: mean value thm} it follows that for any sequence $\lbrace x_n\rbrace_n\subset (-P/2,0)$, and by symmetry indeed any sequence in $(-P/2,P/2)$, such that $x_n\rightarrow 0$,
\begin{equation*}
\liminf_{n\rightarrow \infty} \frac{|\mu-\varphi(x_n)|}{|x_n|} \gtrsim 1.
\end{equation*}
As the sequence was arbitrary this proves (iii). 

Since $\varphi\in L^\infty(\R)$ is symmetric and $\varphi'\geq 0$, and therefore also $L(\varphi^2)'\geq 0$, on $(-P/2,0)$, we have that for $x<0$
\begin{align*}
\left(L(\varphi^2)\right)'(x)  = & \int_\R K'(x-y)(\varphi(y))^2\dy \\
= & \int_{-\infty}^0 (K'(x-y)+K'(x+y))(\varphi(y))^2\dy \\
\leq & \int_x^0(K'(x-y)+K'(x+y))(\varphi(y))^2\dy \\
\leq & C|x|,
\end{align*}
for some constant $C>0$, where we used that $K$ is completely monotone on $(0,\infty)$ and that the integrand is $L^\infty$. The results above imply that $(\mu-\varphi(x))\geq C'|x|$ for some constant $C'$ independent of $x$ when $\varphi(x)>\frac{\mu}{4}$ and from the equation
\begin{equation*}
(\mu-\varphi(x))\varphi'(x)=3\left(L(\varphi^2)\right)'(x)\leq \min(L(\varphi^2)(0), C|x|),
\end{equation*}
which holds for $x\leq 0$, we then see that $\varphi'$ is uniformly bounded on the closed interval $[-P/2,0]$ and therefore Lipschitz. This proves (ii).
\end{proof}

\begin{remark}[On cuspons]
The equality \eqref{eq: equality at x_0} holds when $\varphi(x_0)=\mu$ for any solution of \eqref{eq: DP}, regardless of the integration constant $a$, and we have therefore shown that any $L^\infty$ solution of \eqref{eq: DP} is at least Lipschitz continuous. We have not yet proved that a solution that touches the line $\mu$ exists, but any that do will be Lipschitz. This means that there are no $L^\infty$ cuspons for the DP equation.
\end{remark}

\section{Global bifurcation}
\label{sec: GB}
In this section we will show that there are non-constant periodic solutions which achieve the maximal height; i.e. periodic peakons. These will be obtained by constructing a curve of even, periodic smooth solutions using "standard" bifurcation theory and showing that in the limit of the curve we get a peakon. 

We therefore fix $\alpha\in (1,2)$ and consider $C_{\text{even}}^\alpha(\Si_P)$, the space of even, real-valued functions on the circle $\Si_P$ of finite circumference $P>0$ that are $\floor{\alpha}$-times differentiable with the $\floor{\alpha}$ derivative being $\alpha-\floor{\alpha}$-H\"older continuous. The main point is to work with regularity strictly higher than Lipschitz, i.e. $\alpha>1$, and avoid integer values of $\alpha$ in order to avoid the Zygmund spaces which do not coincide with $C^\alpha$ when $\alpha\in \Z$ (see the proof of Theorem \ref{thm: regularity I}).

From \cite{Lenells2005tws} we know that there are no periodic peakons when $a=0$ in \eqref{eq: DP}, only a one-parameter family of smooth periodic solutions and a peaked solitary wave, and for $a\in (-\frac{\mu^2}{8},0)$ there are only smooth solutions.  As our final goal is to find a bifurcation curve of periodic solutions that converges to a peaked solution, the case $a\leq 0$ is not relevant and henceforth we will only consider $a>0$.

\begin{remark}
As one can easily check (following the procedure below), for $a=0$ one can do local bifurcation from the curve $(\varphi,\mu)=(\mu/2,\mu)$ of constant solutions only when the period is $\sqrt{2}\pi$, but this curve cannot be extended to a global one. When $a\in (-\frac{\mu^2}{8},0)$ all the results regarding bifurcation below holds for periods $0<P<\sqrt{2}\pi$ and we get global bifurcation curves. However, in this case $\sqrt{-8a}<\mu<\infty$ and the equivalent of Lemma \ref{lem: lower bound mu} does not hold. That is, we cannot preclude that alternative (ii) in Theorem \ref{thm: global bif} occurs by $\mu(s)$ approaching $\sqrt{-8a}$. 
\end{remark}

Fix $a>0$ and let $F:C_{\text{even}}^\alpha(\Si_P)\times \R\rightarrow C_\text{even}^\alpha(\Si_P)$ be the operator defined by
\begin{equation}
\label{eq: F 2}
F(\varphi,\mu)=\mu \varphi-\frac{3}{2}L(\varphi^2)-\frac{1}{2}\varphi^2+a.
\end{equation}
Then $\varphi$ is a solution to \eqref{eq: DP} with wave-speed $\mu$ if and only if $F(\varphi,\mu)=0$. There are two curves of constant solutions $F(\varphi(s),\mu(s))=0$, namely $(\varphi(s),\mu(s))=(\frac{s}{4}\pm\frac{\sqrt{s^2+8a}}{4},s)$ for all $s\in \R$. The negative one, however, is not interesting as \eqref{eq: DP} has no non-positive solutions and therefore no curve of non-trivial solutions intersects it. We therefore take the curve $(\varphi(s),\mu(s))=(\frac{s}{4}+\frac{\sqrt{s^2+8a}}{4},s)$ as our starting point. Set
\begin{equation*}
\lambda(\mu):=\frac{\mu}{4}+\frac{\sqrt{\mu^2+8a}}{4}
\end{equation*}
and define
\begin{equation}
\label{eq: tilde F 2}
\tilde{F}(\phi,\mu)=F(\lambda(\mu)-\phi,\mu)=(\lambda-\mu)\phi+3\lambda L(\phi)-\frac{3}{2}L(\phi^2)-\frac{1}{2}\phi^2.
\end{equation}
Then $\tilde{F}(0,\mu)=0$ for all $\mu\in \R$, and letting
\begin{equation}
\label{eq: def varphi}
\varphi:=\lambda(\mu)-\phi,
\end{equation}
we have that
\begin{equation*}
\tilde{F}(\phi,\mu)=0 \Leftrightarrow F(\varphi,\mu)=0.
\end{equation*}
Hence a curve $(\phi(s),\mu(s))$ along which $\tilde{F}=0$ gives rise to a curve of solutions $(\varphi(s),\mu(s))$ to \eqref{eq: DP}. In the sequel, $\varphi$ will always be defined through \eqref{eq: def varphi}.

Note that
\begin{equation*}
\diff_\phi \tilde{F}[0,\mu]=(\lambda(\mu)-\mu) \id +3\lambda(\mu)L.
\end{equation*}
When $\mu^2>a$ we have that $4\lambda>\mu$ while $\mu>\lambda$, and as $L(\cos(p\cdot)(x)=\frac{\cos(px)}{1+p^2}$ we get that
\begin{equation*}
\ker \diff_\phi \tilde{F}[0,\mu]=\lbrace C\cos\left(\sqrt{\frac{4\lambda-\mu}{\mu-\lambda}}x\right) : C\in \R\rbrace.
\end{equation*}
Restricting to $P$-periodic functions, the kernel is one-dimensional if and only if $\sqrt{\frac{4\lambda-\mu}{\mu-\lambda}}=\frac{2k\pi}{P}$ for some $k\in \N$. Clearly, $\sqrt{\frac{4\lambda(\mu)-\mu}{\mu-\lambda(\mu)}}$ is continuous in $\mu$ for $\mu\in( \sqrt{a},\infty)$, strictly monotone on this interval, bounded below by $\sqrt{2}$, the bound being achieved in the limit as $\mu\rightarrow \infty$, and unbounded above as $\mu^2\searrow a$. This means that for every $P>0$ and each $k\in \N$ such that $\frac{2k\pi}{P}>\sqrt{2}$, there exists a unique $\mu>\sqrt{a}$ such that $\cos\left(\sqrt{\frac{4\lambda-\mu}{\mu-\lambda}}x\right)\in C_{\text{even}}^\alpha(\Si_P)$. When $P\geq \sqrt{2}\pi$, we get that $k>1$.
\begin{theorem}[Local bifurcation]
\label{thm: local bif}
Fix $a>0$ and $P>0$, and let $F$ and $\tilde{F}$ be defined as in \eqref{eq: F 2} and \eqref{eq: tilde F 2}, respectively. Then for each $k\in \N$ such that $\frac{2k\pi}{P}>\sqrt{2}$, there exists a unique $\mu_k\in (\sqrt{a},\infty)$ such that $(0,\mu_k)$ is a bifurcation point for $\tilde{F}$, and hence $(\lambda(\mu_k),\mu_k)$ is a bifurcation point for $F$. That is, there exists $\varepsilon>0$ and an analytic curve
\begin{equation*}
s \mapsto (\varphi(s),\mu(s))\subset C_{\text{even}}^\alpha(\Si_P)\times (\sqrt{a},\infty), \quad |s|<\varepsilon,
\end{equation*}
of nontrivial $P/k$-periodic solutions, where $\mu(0)=\mu_k$ and 
\begin{equation*}
\diff_s\phi(0)=-\diff_s\varphi(0)=\cos\left(\sqrt{\frac{4\lambda(\mu_k)-\mu_k}{\mu_k-\lambda(\mu_k)}}x\right).
\end{equation*}

\end{theorem}
\begin{proof}
It is sufficient to consider $k=1$ and $P<\sqrt{2}\pi$. As shown above, there exists a unique $\mu\in (\sqrt{a},\infty)$ such that $\ker \diff_\phi \tilde{F}[0,\mu]$ is one-dimensional. The space $C_{\text{even}}^\alpha(\Si_P)$ has basis $\lbrace \cos(\frac{2\pi}{P} k\cdot) : k\in \N\rbrace$ and by straightforward calculation one finds that $\diff_\phi \tilde{F}[0,\mu]$ maps the basis element $k=1$ to zero while all others are preserved modulo a constant. Thus $\text{codim}\, \text{range} \,\diff_\phi \tilde{F}[0,\mu]=1$ and $\diff_\phi \tilde{F}[0,\mu]$ is Friedholm of index zero. The result now follows from Theorem 8.3.1 in \cite{Buffoni2003ato}. Note that $\diff_s\phi(0)=-\diff_s\varphi(0)$ because $D_s \mu(0)=\dot{\mu}(0)=0$ (see \eqref{eq: dot mu} below).
\end{proof}
We want to extend these bifurcation curves globally. Let 
\begin{equation*}
U:=\lbrace (\varphi,\mu)\in C_{\text{even}}^\alpha(\Si_P)\times (\sqrt{a},\infty) : \varphi<\mu \rbrace,
\end{equation*}
and
\begin{equation*}
S:=\lbrace (\varphi,\mu)\in U : F(\varphi,\mu)=0\rbrace.
\end{equation*}
In order to establish Theorem \ref{thm: global bif} below; that is, to extend the curves globally, it suffices to establish that $\ddot{\mu}(0)\neq 0$ and the following Lemma:
\begin{lemma}
\label{lem: compact subsets}
Whenever $(\varphi,\mu)\in S$ the function $\varphi$ is smooth, and bounded and closed subsets of $S$ are compact in $C_{\text{even}}^\alpha(\Si_P)\times (\sqrt{a},\infty)$.
\end{lemma}
\begin{proof}
The smoothness part was proved in Theorem \ref{thm: regularity I}. Recall from the proof of that theorem that $(\varphi,\mu)\in S$ implies $3L(\varphi^2)<\mu^2+2a$ and hence
\begin{equation*}
\varphi=\mu-\sqrt{\mu^2+2a-3L(\varphi^2)}\in C_{\text{even}}^{\alpha+2}(\Si_P),
\end{equation*}
as $L:C^\alpha\rightarrow C^{\alpha+2}$ and $\sqrt{x}$ is real analytic for $x>0$. Let $E\subset S$ be bounded and closed in the $C_{\text{even}}^\alpha(\Si_P)\times\R$ topology. Then, as shown above, $\lbrace \varphi : (\varphi,\mu)\in E\rbrace\subset C_{\text{even}}^{\alpha+2}(\Si_P)$ is a bounded subset. Bounded subsets of $C_{\text{even}}^{\alpha+2}(\Si_P)$ are pre-compact in $C_{\text{even}}^{\alpha}(\Si_P)$, hence any sequence $\lbrace (\varphi_n,\mu_n)\rbrace_n\subset E$ has a subsequence that converges in the $C_{\text{even}}^\alpha(\Si_P)\times\R$ topology. As $E$ is closed, the limit must itself lie in $E$, proving that $E$ is compact.
\end{proof}

In order to establish the bifurcation formulas we will apply the Lyapunov-Schmidt reduction \cite{Kielhofer2012bta}. For simplicity we consider the case $P<\sqrt{2}\pi$ and $k=1$. Let $\mu^*:=\mu_1$ and
\begin{equation}
\phi^*(x):=\cos\left( \frac{2\pi}{P}x\right),
\end{equation}
and let furthermore
\begin{equation*}
M:=\lbrace \sum_{k\neq 1} a_k\cos\left( \frac{2\pi kx}{P}\right)\in C_\text{even}^\alpha (\Si_P)\rbrace,
\end{equation*}
and
\begin{equation*}
N:=\ker \diff_\phi \tilde{F}[0,\mu^*]=\text{span}(\phi^*).
\end{equation*}
Then $C_{\text{even}}^\alpha (\Si_P)=M\oplus N$ and we can use the canonical embedding $C^\alpha(\Si_P)\hookrightarrow L^2(\Si_P)$ to define a continuous projection
\begin{equation}
\Pi \phi=\langle \phi,\phi^*\rangle_{L^2(\Si_P)} \phi^*,
\end{equation}
where $\langle u,v\rangle_{L^2(\Si_P)}=\frac{2}{P}\int_{-P/2}^{P/2}uv\dx$.

\begin{theorem}[Lyapunov-Schmidt reduction \cite{Kielhofer2012bta}]
There exists a neighbourhood $O\times Y\subset U$ around $(0,\mu^*)$ in which the problem
\begin{equation}
\label{eq: inf dim}
\tilde{F}(\phi,\mu)=0
\end{equation}
is equivalent to
\begin{equation}
\label{eq: finite dim}
\Phi(\varepsilon \phi^*,\mu):= \Pi \tilde{F}(\varepsilon\phi^* +\psi(\varepsilon\phi^*,\mu),\mu)=0
\end{equation}
for functions $\psi\in C^\infty(O_N\times Y,M)$, $\Phi\in C^\infty(O_N\times Y,N)$, and $O_N\subset N$ an open neighbourhood of the zero function in $N$. One has $\Phi(0,\mu^*)=0$, $\psi(0,\mu^*)=0$, $\diff_\phi \psi(0,\mu^*)=0$, and solving the finite dimensional problem \eqref{eq: finite dim} provides a solution $\phi=\varepsilon\phi^*+\psi(\varepsilon \phi^*,\mu)$ to the infinite dimensional problem \eqref{eq: inf dim}.
\end{theorem}
We want to show that $\mu(\varepsilon)$ is not constant around $0$. We calculate
\begin{align*}
\diff_{\phi\phi}^2 \tilde{F}[0,\mu^*](\phi^*,\phi^*) & =-(\phi^*)^2-3L((\phi^*)^2), \\
\diff_{\mu\phi}^2\tilde{F}[0,\mu^*]\phi^* & =(\lambda'(\mu^*)-1)\phi^*+3\lambda'(\mu^*)L(\phi^*).
\end{align*}
As  $L(\cos(p\cdot))(x)=\frac{1}{1+p^2}\cos(px)$ for $p\neq 0$, we get that
\begin{equation*}
\diff_{\mu\phi}^2\tilde{F}[0,\mu^*]\phi^*=\left( \lambda'(\mu^*)(1+\frac{3}{1+(2\pi/P)^2})-1\right)\phi^*.
\end{equation*}
By choice, $\sqrt{\frac{4\lambda(\mu^*)-\mu^*}{\mu^*-\lambda(\mu^*)}}=\frac{2\pi}{P}$, so that the coefficient of $\phi^*$ above is zero if and only if
\begin{equation*}
\lambda'(\mu^*)=\frac{\lambda(\mu^*)}{\mu^*}.
\end{equation*}
This is impossible, as the left-hand side lies in $(\frac{1}{3},\frac{1}{2})$ when $\mu^*\in(\sqrt{a},\infty)$, while the right-hand side lies in $(\frac{1}{2},1)$.

Using bifurcation formulas (see e.g. section I.6 in \cite{Kielhofer2012bta}), we readily calculate $\dot{\mu}(0)$:
\begin{equation}
\label{eq: dot mu}
\dot{\mu}(0)=-\frac{1}{2}\frac{\langle \diff_{\phi\phi}^2 \tilde{F}[0,\mu^*](\phi^*,\phi^*),\phi^*\rangle_{L^2(\Si_P)} }{\langle \diff_{\mu\phi}^2\tilde{F}[0,\mu^*]\phi^*,\phi^* \rangle_{L^2(\Si_P)} }=0,
\end{equation}
as $\int_{-P/2}^{P/2} \cos^3(\frac{2\pi}{P}x)\dx=0$.
When $\dot{\mu}(0)=0$, one has that (\cite{Kielhofer2012bta})
\begin{equation*}
\ddot{\mu}(0)=-\frac{1}{3}\frac{\langle \diff_{\phi\phi\phi}^3 \Phi[0,\mu^*](\phi^*,\phi^*,\phi^*),\phi^*\rangle_{L^2(\Si_P)}}{\langle \diff_{\mu\phi}^2\tilde{F}[0,\mu^*]\phi^*,\phi^* \rangle_{L^2(\Si_P)} }.
\end{equation*}
The denominator equals $\left( \lambda'(\mu^*)(1+\frac{1}{1+(2\pi/P)^2})-1\right)\neq 0$. Using that $\tilde{F}$ is quadratic in $\phi$, one can calculate that
\begin{align*}
& \diff_{\phi\phi\phi}^3 \Phi[\phi,\mu](\phi^*,\phi^*,\phi^*) \\
& = 3\,\Pi\, \diff_{\phi\phi}^2 \tilde{F}[\phi+\psi(\phi,\mu),\mu](\phi^*+\diff_\phi \psi[\phi,\mu]\phi^*,\diff_{\phi\phi}^2 \psi[\phi,\mu](\phi^*,\phi^*)) \\
& \quad + \Pi\, \diff_\phi \tilde{F}[\phi+\psi(\phi,\mu),\mu] \diff_{\phi\phi\phi}^3 \psi[\phi,\mu](\phi^*,\phi^*,\phi^*).
\end{align*}
As $N=\ker \diff_\phi\tilde{F}[0,\mu^*]$, we get that the projection $\Pi\, \diff_\phi\tilde{F}[0,\mu^*]=0$. Using that $\psi(0,\mu^*)=\diff_\phi \psi[0,\mu^*]=0$ and the expression for $\diff_{\phi\phi}^2 \tilde{F}[0,\mu^*]$ above, we find that
\begin{align}
\diff_{\phi\phi\phi}^3 & \Phi[0,\mu^*](\phi^*,\phi^*,\phi^*) \nonumber \\
&= -\Pi\, \left(\phi^* \diff_{\phi\phi}^2 \psi[0,\mu^*](\phi^*,\phi^*)+3L(\phi^* \diff_{\phi\phi}^2 \psi[0,\mu^*](\phi^*,\phi^*))\right). \label{eq: third derivative varphi}
\end{align}
We can rewrite $\diff_{\phi\phi}^2 \psi[0,\mu^*](\phi^*,\phi^*)$ as
\begin{align*}
\diff_{\phi\phi}^2  \psi[0,\mu^*](\phi^*,\phi^*) & =-\left( \diff_\phi \tilde{F}[0,\mu^*]\right)^{-1}(\id-\Pi)\diff_{\phi\phi}^2 \tilde{F}[0,\mu^*](\phi^*,\phi^*) \\
& =\left( \diff_\phi \tilde{F}[0,\mu^*]\right)^{-1}\left( (\phi^*)^2+3L((\phi^*)^2)\right) \\
&=\left( \diff_\phi \tilde{F}[0,\mu^*]\right)^{-1}\left( 2+(\frac{1}{2}+\frac{3 P^2}{16\pi^2+P^2})\cos\left( \frac{4\pi}{P}x\right)\right) \\
& =\frac{2}{\lambda(\mu^*)-\mu^*} \\
&\quad +\frac{16\pi^2+7P^2}{2((4\lambda(\mu^*)-\mu^*)P^2+16\pi^2(\lambda(\mu^*)-\mu^*)}\cos\left( \frac{4\pi}{P}x\right),
\end{align*}
where we used that $L(\cos(p\cdot))(x)=\frac{1}{1+p^2}\cos(px)$ for $p\neq 0$. Multiplying with $\phi^*(x)=\cos\left(\frac{2\pi}{P}x\right)$ and using double and triple angle formulas, we get
\begin{align*}
\frac{2\cos(2\pi x/P)}{\lambda(\mu^*)-\mu^*} & +\frac{1}{2}\frac{16\pi^2+7P^2}{2((4\lambda(\mu^*)-\mu^*)P^2+16\pi^2(\lambda(\mu^*)-\mu^*)}\cos\left(\frac{2\pi}{P}x\right)\\
& +\frac{1}{2}\frac{16\pi^2+7P^2}{2((4\lambda(\mu^*)-\mu^*)P^2+16\pi^2(\lambda(\mu^*)-\mu^*)}\cos\left(\frac{6\pi}{P}x\right).
\end{align*}
Denoting by $C$ be the coefficient of $\cos\left( \frac{2\pi}{P}x\right)=\phi^*(x)$ in the above expression, we see from \eqref{eq: third derivative varphi} that
\begin{equation*}
\diff_{\phi\phi\phi}^3  \Phi[0,\mu^*](\phi^*,\phi^*,\phi^*)=-C\left(1+\frac{3P^2}{P^2+4\pi^2}\right)\phi^*.
\end{equation*}
Hence $\ddot{\mu}(0)\neq 0$ and $\dot{\mu}\not \equiv 0$ on $(-\varepsilon,\varepsilon)$. Lemma \ref{lem: compact subsets} and the calculations above show that the conditions of Theorem 9.1.1 in \cite{Buffoni2003ato} are fulfilled and we have the following result:

\begin{theorem}[Global bifurcation]
\label{thm: global bif}
The local bifurcation curves $s \mapsto (\varphi(s),\mu(s))$ of solutions to the Degasperis-Procesi equation from Theorem \ref{thm: local bif} extend to global continuous curves $\mathfrak{R}$ of solutions $\R_{\geq 0}\rightarrow S$. One of the following alternatives hold:
\begin{itemize}
\item[(i)] $\|(\varphi(s),\mu(s))\|_{C^\alpha(\Si_P)\times \R}\rightarrow \infty$ as $s\rightarrow \infty$.
\item[(ii)] $(\varphi(s),\mu(s))$ approaches the boundary of $U$ as $s\rightarrow \infty$.
\item[(iii)] The function $s \mapsto (\varphi(s),\mu(s))$ is (finitely) periodic.
\end{itemize}
\end{theorem}

\begin{theorem}
\label{thm: alternative (iii)}
Alternative (iii) in Theorem \ref{thm: global bif} cannot occur.
\end{theorem}

\begin{proof}
Let 
\begin{equation*}
\mathcal{K}:=\lbrace \varphi\in C_{\text{even}}^\alpha(\Si_P) : \varphi \,\, \text{is non-decreasing on} \,\, (-P/2,0)\rbrace,
\end{equation*}
which is a closed cone in $C^\alpha(\Si_P)$, and let $\mathfrak{R}^1$ and $S^1$ denote the $\varphi$ parts of $\mathfrak{R}$ and $S$ respectively. The result follows from Theorem 9.2.2 in \cite{Buffoni2003ato} if we can show that if $\varphi\in \mathrm{R}^1\cap \mathcal{K}$ is non-constant, then $\varphi$ is an interior point of $S^1\cap \mathcal{K}$. To see this, let $\varphi$ be a non-constant solution that is non-decreasing on $(-P/2,0)$. By Theorem \ref{thm: regularity I}, $\varphi$ is smooth and we can apply Theorem \ref{thm: strictly increasing} to conclude that $\varphi''(0)<0$, $\varphi''(-P/2)>0$ and $\varphi'>0$ on $(-P/2,0)$. Let $\psi$ be a solution within $\delta\ll 1$ distance of $\varphi$ in $C^\alpha$, with $\delta$ small enough that $\psi<\mu$. Iterating as in the proof of Theorem \ref{thm: regularity I}, we get that $\|\varphi-\psi\|_{C^2}<\tilde{\delta}$, where $\tilde{\delta}$ can be made arbitrarily small by taking $\delta$ smaller. This implies that $\psi$ also is non-decreasing on $(-P/2,0)$. Hence $\psi\in S^1\cap \mathcal{K}$.
\end{proof}

\begin{lemma}
\label{lem: uniform convergence}
Any sequence $\lbrace(\varphi_n,\mu_n)\rbrace_n\subset S$ of solutions to \eqref{eq: DP} with $\lbrace \mu_n\rbrace_n$ bounded has a subsequence that converges uniformly to a solution $\varphi$.
\end{lemma}

\begin{proof}
From \eqref{eq: DP} we have that
\begin{equation*}
\frac{1}{2}\varphi^2=a+\mu \varphi-\frac{3}{2}L(\varphi^2)<a+\mu\varphi,
\end{equation*}
which implies that
\begin{equation*}
\|\varphi\|_{L^\infty}^2\leq 2a+2\mu\|\varphi\|_{L^\infty}.
\end{equation*}
Hence $\lbrace\varphi_n\rbrace_n$ is bounded whenever $\lbrace \mu_n\rbrace_n$ is. We have that 
\begin{align*}
|L(\varphi_n^2)(x+h)-L(\varphi_n^2)(x)|& =\left|\int_\R (K(x+h-y)-K(x-y))\varphi_n(y)^2 \dy\right| \\
& \leq \|\varphi_n\|_{L^\infty}^2\int_\R |K(x+h-y)-K(x-y)|\dy.
\end{align*}
As $K$ is continuous and integrable, the final integral can be made arbitrarily small by taking $h$ sufficiently small. This shows that $\lbrace L(\varphi_n^2)\rbrace_n$ is equicontinuous. Arzela-Ascoli's theorem then implies the existence of a uniformly convergent subsequence.
\end{proof}

\begin{lemma}
\label{lem: lower bound mu}
For fixed $a>0$ and $P>0$, $\mu(s)$ does not approach $\sqrt{a}$ as $s\rightarrow \infty$.
\end{lemma}
\begin{proof}
Assume for a contradiction that there is a sequence $\lbrace \mu_n\rbrace_n$ such that $\mu_n\rightarrow \sqrt{a}$ as $n\rightarrow \infty$, while at the same time $\varphi_n=\varphi_{\mu_n}$ is a sequence along the global bifurcation curve in Theorem \ref{thm: global bif}. According to Lemma \ref{lem: uniform convergence} a subsequence $\lbrace \varphi_{n_k}\rbrace_k$ converges to a solution $\varphi_0$ of \eqref{eq: DP}. From Theorem \ref{thm: solutions to DP} we have that $\max \varphi_{n_k}>\frac{\mu_{n_k}+\sqrt{\mu_{n_k}^2+8a}}{4}>\sqrt{a}$, while $\max \varphi_{n_k}<\mu_{n_k}\rightarrow \sqrt{a}$. It follows that $\max \varphi_0=\sqrt{a}$ and hence $\max L(\varphi_0^2)=a$. However, $\max L(\varphi^2)\leq \max \varphi^2$ with equality if and only if $\varphi$ is constant. Hence $\varphi_0 \equiv \sqrt{a}$. This leads to a contradiction with Lemma \ref{lem: global bound}, noting that the constant $C_{K,P,\mu}$ is positive for all positive $\mu$, as we get that
\begin{equation*}
0=\lim_{k\rightarrow \infty} \mu_{n_k}-\varphi_{n_k}(P/2)\geq \lim_{k\rightarrow \infty} C_{K,P,\mu_{n_k}}>0.
\end{equation*}
\end{proof}

\begin{lemma}
\label{lem: alt (i) and (ii)}
Let $a>0$ and $P>0$. If $\sup_{s\geq 0} \mu(s)<\infty$, then alternatives (i) and (ii) in Theorem \ref{thm: global bif} both occur.
\end{lemma}
\begin{proof}
We already know from Theorem \ref{thm: alternative (iii)} that alternative (iii) cannot occur, thus either (i), (ii), or both has to occur. Theorem \ref{thm: regularity II} implies that alternative (i) happens if $\lim_{s\rightarrow \infty}\mu(s)-\varphi(s)(0)=0$. From 
\begin{equation*}
(\mu-\varphi)\varphi'=\frac{3}{2}\left( L(\varphi^2)\right)'\leq \frac{3}{2}L(\varphi^2),
\end{equation*}
we see that $\varphi'$ is bounded in $\mu$. Similarly, it is easy to see that if $\varphi(0)<\mu$, then $\|\varphi\|_{C^2 (\Si_P)}$ is bounded in $\mu$. Hence, if $\sup_{s\geq 0} \mu(s)<\infty$, alternative (i) happens if and only if $\lim_{s\rightarrow \infty}\mu(s)-\varphi(s)(0)=0$, which implies that (ii) occurs as well.

From Lemma \ref{lem: lower bound mu} we know that $\inf_{s\geq 0} \mu(s)>\sqrt{a}$ and the assumption $\sup_{s\geq 0} \mu(s)<\infty$ then implies that $\mu(s)$ does not approach the boundary of $(\sqrt{a},\infty)$. Thus alternative (ii) can only happen if $\lim_{s\rightarrow \infty}\mu(s)-\varphi(s)(0)=0$, which in turn implies (i).
\end{proof}


\begin{proposition}
\label{prop: upper bound mu}
For fixed $a>0$, there is a number $C>0$ such that if $P<C$, there is an upper bound on $\mu$ above which there are no smooth solutions to \eqref{eq: DP} except constant solutions.
\end{proposition}
\begin{proof}
Assume $\varphi$ is a smooth solution to \eqref{eq: DP} which is even and non-decreasing on $(-P/2,0)$ (recall that $\varphi$ is smooth if $\varphi(0)<\mu$, and a peakon if $\varphi(0)=\mu$; no other possibilities exists).
We know that $\varphi'$ has a maximum on $(-P/2,0)$, say $\varphi'(x_0)=\max \varphi'$. Then $\varphi''(x_0)=0$. As
\begin{equation*}
(\mu-\varphi(x))\varphi''(x)=(\varphi'(x))^2+3\int_{-P/2}^0(K_P'(x-y)-K_P'(x+y))\varphi(y)\varphi'(y)\dy,
\end{equation*}
and $K_P'(x-y)-K_P'(x+y)>0$ for $x<y<0$, we then get that
\begin{align*}
(\varphi'(x_0))^2&=-3\int_{-P/2}^0(K_P'(x_0-y)-K_P'(x_0+y))\varphi(y)\varphi'(y)\dy \\
&\leq -3\int_{-P/2}^{x_0}(K_P'(x_0-y)-K_P'(x_0+y))\varphi(y)\varphi'(y)\dy \\
& = \varphi(c_0)\varphi'(c_0)3\left|\int_{-P/2}^{x_0}(K_P'(x_0-y)-K_P'(x_0+y))\dy\right|,
\end{align*}
where $-P/2<c_0<x_0$. As $\varphi'(c_0)<\varphi'(x_0)$ and $\varphi(c_0)<\mu$, it follows that $\max \varphi'<C_P \mu$, where the constant $C_P$ depends on $P$ through the final integral above. As $K_P(x)=\frac{1}{2}\mathrm{e}^{-|x|}+\frac{\cosh(x)}{\mathrm{e}^P-1}$, the derivative is bounded by $1/2$ and the final integral above, hence also $C_P$, therefore goes to $0$ as $P\rightarrow 0$. From Theorem \ref{thm: solutions to DP}, we know that a solution $\varphi$ satisfies
\begin{equation*}
\min \varphi  <\frac{\mu+\sqrt{\mu^2+8a}}{4}<\max  \varphi .
\end{equation*}
If $\mu \ggg a$, then $\frac{\mu+\sqrt{\mu^2+8a}}{4}=\frac{\mu}{2}+\bigO(\mu^{-1})$. Hence there exists a point $x_1\in (-P/2,0)$ such that $\varphi(x_1)=\frac{\mu}{2}+\bigO(\mu^{-1})$. Trivially, for every $x\in (-P/2,0)$ we have the bounds
\begin{equation*}
\varphi(x_1)-(P/2)\max \varphi'<\varphi(x)<\varphi(x_1)+(P/2)\max \varphi'.
\end{equation*}
Combining this with the bound on the derivative above, we get
\begin{equation*}
\max \varphi<\frac{\mu}{2}+\frac{P}{2}C_P \mu+\bigO(\mu^{-1}).
\end{equation*}
For any $c\in (\frac{1}{2},1)$ we can take $P>0$ sufficiently small independently of $\mu$ such that
\begin{equation}
\label{eq: upper bound}
\max \varphi\leq c\mu +\bigO(\mu^{-1}).
\end{equation}
By the mean value theorem,
\begin{align*}
(\mu-\varphi(x))\varphi'(x) & =\frac{3}{2}\left(L(\varphi^2)\right)'(x) \\
& =3\int_{-P/2}^0(K_P(x-y)-K_P(x+y))\varphi'(y)\varphi(y)\dy \\
& =3 \varphi'(c_x)\varphi(c_x)\int_{-P/2}^0(K_P(x-y)-K_P(x+y))\dy,
\end{align*}
for some constant $c_x$ that depends on $x$. From \eqref{eq: upper bound} we get that there is a constant $C$ independent of $\mu$ and decreasing in $P$ such that $\varphi(c_x)/(\mu-\varphi(x))\leq C+\bigO(\mu^{-2})$ for all $x\in (-P/2,0)$. We therefore get that
\begin{equation}
\label{eq: upper bound derivative}
\varphi'(x)\leq \varphi'(c_x)C\int_{-P/2}^0(K_p(x-y)-K_P(x-y))\dy+\bigO(\mu^{-1}),
\end{equation}
where $C$ is independent of $\mu$ and decreases with $P$. The integral on the right hand side goes to $0$ for all $x\in (-P/2,0)$ as $P\rightarrow 0$. For $P$ sufficiently small, \eqref{eq: upper bound derivative} implies that $\varphi'\equiv 0$ for all sufficiently large $\mu$.
\end{proof}

\begin{theorem}
Let $a>0$ be fixed. For all $P>0$ sufficiently small, alternatives (i) and (ii) in Theorem \ref{thm: global bif} both occur. Given any unbounded sequence of positive numbers $s_n$, a subsequence of $\lbrace \varphi(s_n)\rbrace_n$ converges uniformly to a limiting wave $\varphi$ that solves \eqref{eq: DP} and satisfies
\begin{equation*}
\varphi(0)=\mu, \quad \varphi\in C^{0,1}(\R).
\end{equation*}
The limiting wave is even, strictly increasing on $(-P/2,0)$ and is exactly Lipschitz at $x\in P\Z$.
\end{theorem}
\begin{proof}
From Theorem \ref{thm: alternative (iii)}, we know that alternative (iii) cannot occur. The proof of Theorem \ref{thm: alternative (iii)} also implies that the curve $(\varphi(s),\mu(s))$ cannot reconnect to the curve of constant solutions we bifurcated from for any finite $s$. Hence Proposition \ref{prop: upper bound mu} implies that for all $P>0$ sufficiently small, $\sup_{s\geq 0} \mu(s)<\infty$, and by Lemma \ref{lem: alt (i) and (ii)} we get that alternatives (i) and (ii) both occur. Moreover, as $\lbrace \mu(s_n)\rbrace_n$ is bounded, Lemma \ref{lem: uniform convergence} gives that a subsequence of $\lbrace \varphi(s_n)\rbrace_n$ converges uniformly to a solution $\varphi$. As alternatives (i) and (ii) both occur, this solution must necessarily have the stated properties.
\end{proof}

\medskip
\noindent

\bibliographystyle{plain}
\bibliography{cusped_waves}

\begin{thebibliography}{10}

\bibitem{Buffoni2003ato}
B.~Buffoni and J.~Toland.
\newblock {\em Analytic Theory of Global Bifurcation}.
\newblock Princeton University Press, Princeton, New Jersey, 2003.

\bibitem{Camassa1993ais}
R.~Camassa and D.~D. Holm.
\newblock An integrable shallow water equation with peaked solitons.
\newblock {\em Phys. Rev. Lett.}, 71(11):1661--1664, 1993.

\bibitem{Constantin2009thr}
A.~Constantin and D.~Lannes.
\newblock The hydrodynamical relevance of the {C}amassa-{H}olm and
  {D}egasperis-{P}rocesi equations.
\newblock {\em Arch. Ration. Mech. Anal.}, 192(1):165--186, 2009.

\bibitem{Degasperis2002ani}
A.~Degasperis, D.~D. Holm, and A.~N.~I. Khon.
\newblock A new integrable equation with peakon solutions.
\newblock {\em Teoret. Mat. Fiz.}, 133(2):170--183, 2002.

\bibitem{Degasperis1999ai}
A.~Degasperis and M.~Procesi.
\newblock Asymptotic integrability.
\newblock In {\em Symmetry and perturbation theory ({R}ome, 1998)}, pages
  23--37. 1999.

\bibitem{Ehrnstrom2016eoa}
M.~Ehrnstr\"om, M.~A. Johnson, and K.~M. Claassen.
\newblock Existence of a highest wave in a fully dispersive two-way shallow
  water model.
\newblock arXiv:1610.02603, 2016.

\bibitem{Ehrnstrom2016owc}
M.~Ehrnstr\"om and E.~Wahl{\'e}n.
\newblock On {W}hitham's conjecture of a highest cusped wave for a nonlocal
  dispersive equation.
\newblock arXiv:1602.05384, 2016.

\bibitem{Escher2006gws}
J.~Escher, Y.~Liu, and Z.~Yin.
\newblock Global weak solutions and blow-up structure for the
  {D}egasperis-{P}rocesi equation.
\newblock {\em J. Funct. Anal.}, 241(2):457--485, 2006.

\bibitem{Kielhofer2012bta}
H.~Kielh\"ofer.
\newblock {\em Bifurcation theory}, volume 156 of {\em Applied Mathematical
  Sciences}.
\newblock Springer, New York, second edition, 2012.
\newblock An introduction with applications to partial differential equations.

\bibitem{Lenells2005tws}
J.~Lenells.
\newblock Traveling wave solutions of the {D}egasperis-{P}rocesi equation.
\newblock {\em J. Math. Anal. Appl.}, 306:72--82, 2005.

\bibitem{Whitham1967vma}
G.~B. Whitham.
\newblock Variational methods and applications to water waves.
\newblock {\em Proc. R. Soc. Lond., Ser. A}, 299:6--25, 1967.

\bibitem{Yin2003gef}
Z.~Yin.
\newblock Global existence for a new periodic integrable equation.
\newblock {\em J. Math. Anal. Appl.}, 283(1):129--139, 2003.

\bibitem{Yin2003otc}
Z.~Yin.
\newblock On the {C}auchy problem for an integrable equation with peakon
  solutions.
\newblock {\em Illinois J. Math.}, 47(3):649--666, 2003.

\bibitem{Yin2004gws}
Z.~Yin.
\newblock Global weak solutions for a new periodic integrable equation with
  peakon solutions.
\newblock {\em J. Funct. Anal.}, 212(1):182--194, 2004.

\bibitem{Zhang2007cas}
G.~Zhang and Z.~Qiao.
\newblock Cuspons and smooth solitons of the {D}egasperis-{P}rocesi equation
  under inhomogeneous boundary condition.
\newblock {\em Math. Phys. Anal. Geom.}, 10(3):205--225, 2007.

\end{thebibliography}

\end{document}